\documentclass{IEEEtran}
\usepackage{amsmath,graphicx}

\usepackage{amsthm}
\usepackage{amssymb}
\theoremstyle{plain}
\usepackage{algpseudocode}
\usepackage{algorithm}
\newtheorem{theorem}{Theorem}
\newtheorem{proposition}[theorem]{Proposition}
\newtheorem{corollary}[theorem]{Corollary}
\newtheorem{lemma}[theorem]{Lemma}
\theoremstyle{definition}
\newtheorem{definition}[theorem]{Definition}

\newtheorem{remark}[theorem]{Remark}
\newtheorem{example}[theorem]{Example}

\def\col{\mathcal}

\title{Tracking before detection using partial orders and optimization}

\begin{document}

%
\author{Michael~Robinson$^{(1)}$~\IEEEmembership{Member,~IEEE}, \and Michael~Stein$^{(2)}$~\IEEEmembership{Member,~IEEE}, \and Henry~S.~Owen$^{(3)}$~\IEEEmembership{Member,~IEEE}
\thanks{(1) Mathematics and Statistics, American University, Washington, DC, USA email: michaelr@american.edu
  (2) Tesseract Analytics, LLC
  (3) H.S. Owen, LLC}}

\markboth{}{Tracking before detection using partially ordered sets}

\maketitle

\begin{abstract}
This article addresses the problem of multi-object tracking by using a non-deterministic model of target behaviors with hard constraints.  To capture the evolution of target features as well as their locations, we permit objects to lie in a general topological target configuration space, rather than a Euclidean space.  We obtain tracker performance bounds based on sample rates, and derive a flexible, agnostic tracking algorithm.  We demonstrate our algorithm on two scenarios involving laboratory and field data.
\end{abstract}


\section{Introduction}


This article addresses the problem of tracking multiple objects inhabiting a general configuration space (not just Euclidean space).  We assume that the objects satisfy set-valued dynamical laws that incompletely constrain their behavior.  Set-valued dynamical systems do not have a definite ``internal state,'' but rather have a \emph{set} of indistinguishable possible states.  Set-valued dynamical systems can be used to model objects that make indepedent decisions that are fundamentally impossible to anticipate beforehand.  We desire performance bounds on tracker performance (specifically the number of track drops, the distribution of track lengths, and a measure of object-track association accuracy) that are determined by the available data rates and data quality.


Modern tracking algorithms typically make use of probabilistic state spaces to model the motion of targets in the environment and to model the acquisition of information about these targets.  Without fairly strong features, estimating the parameters describing targets and sensors is difficult.  This becomes particularly acute in highly cluttered or observation-poor environments, because the assumptions of probabilistic models (especially independence) are hard to ensure in practice.  This article addresses the situation of multi-target tracking using coarse, topological constraints.  The algorithms described in this article attempt to minimally constrain target motion by observations.

This article leverages loose kinematic-like constraints through set-valued functions on targets on an abstract topological space to represent both physical location as well as target attributes.  These constraints naturally lead to an algorithmic approach based on the continuity of a target's path through the topological space.  The general framework proposed in this article permits two kinds of constraints to be mixed in a flexible way, namely
\begin{enumerate}
  \item Constraints on target paths based on attributes (eg. kinematics)
  \item Constraints on target attributes based on target paths (eg. aspect angles or textural features)
\end{enumerate}
Observations of the set of targets are then threaded together via the constraints to produce a representation of \emph{possible} target locations---without specifically committing to any inferences about the behavior of a specific target until many more observations have been made.  This allows for a more global optimization of track quality.

\section{Historical discussion}


Multi-object (or multi-target) tracking has been studied extensively for many years.  There are a number of good surveys, such as \cite{Vo_2015,Mahler_2014,Luo_2014,Yang_2011,Pulford_2005} (and many others) that cover the basic probabilistic and deterministic methods.  Various systematic experimental campaigns have also been described \cite{leal2015motchallenge,Solera_2015}.  Other authors show that fusing detections across sensors \cite{Bailer_2014,Mahler_2013,smith2006approaches,Newman_2013,hall2004mathematical} yields better coverage and performance.  Considerable effort is expended developing robust features, though usually the metric for which selecting features is pragmatic rather than theoretical.

The theory of set-valued functions is a natural foundation for tracking algorithms, since they retain multiple, competing hypotheses about object trajectories \cite{Streit_1995}.  Set-valued functions are a common tool in the dynamical systems literature, where they are used in the theory of differential inclusions \cite{Aubin_2012}, control theory \cite{goebel2009hybrid}, economics \cite{cherene2012set}, and robotics \cite{martinez2007motion,erdmann1988exploration}.  

Several authors (most notably \cite{Luo_2014,Pulford_2005}) provide exhaustive taxonomies of tracker algorithms which include probabilistic and deterministic trackers.  Probabilistic methods based on optimal Bayesian updates \cite{Vo_2013,DeGroot_2004} (or their many variations (for instance \cite{hamid2015joint, Deming_2009,Efe_2004}) prefer some \emph{detections} over others.  We note that our approach instead weights \emph{trajectories}.  Although Kalman filters cannot handle hard constraints, particle filters can be tailored to handle hard constraints \cite{Elfring2021ParticleFA}.

Trackers based on optimal network flows tend to be extremely effective in video object tracking \cite{Butt_2013,Berclaz_2011,Pirsiavash_2011}.  Network flow trackers are optimal when target behaviors are probabilistic \cite{Zhang_2008,Perera_2006}, but no sampling rate bounds appear to be available.  There are a number of variations of the basic optimal network flow algorithm, such as providing local updates to the flow once solved \cite{Milan_2015}, using tracklets as detections \cite {wang2014tracklet}, or $K$-shortest paths, sparsely enriched with target feature information \cite{Ben_2014}.  

Our algorithm is a kind of a ``track-before-detect'' algorithm, though not as is usually understood.  While the algorithm uses potential detections as input, those detections that cannot cohere into tracks are ultimately rejected.  Although successful \cite{Yan_2012}, track-before-detect algorithms are not as popular as other approaches that rely on hard detections. Track before detection algorithms usually predict where an object will be after it is observed.  This is generally done geometrically \cite{Possegger_2014}, and often probabilistically \cite{choi2013general}.  In contrast, although our algorithm works with detections, it treats them as \emph{tentative}, and only uses those that cohere into tracks.

Our algorithmic approach is inspired by the Ford-Fulkerson algorithm, but optimized over vertices in addition to edges.  Similar algorithmic approaches have been discussed extensively in the mathematical literature for application to graph and poset-theoretic problems \cite{Bogart_2013}.  Shortly after the Ford-Fulkerson algorithm was published, Norman and Rabin \cite{Norman_1959} discussed the relationship between minimal matching---which our algorithm computes---and maximal covers.  This does not address the problem of \emph{finding} either a matching or a cover, since they worked under the hypotheses of the Ford-Fulkerson algorithm.  Cameron \cite{Cameron_1985} gives a direct algorithmic approach for solving minimal cover problems via linear programming.

\section{Contributions}


In high-resolution video tracking, rich textural features can be derived since there are many pixels on each object.  This helps a netflow tracker disambiguate the path of an object through the network of detections.  The probabilistic proof of optimality of the netflow tracker in \cite{Zhang_2008} relies on a Markov chain to model target motion which further disambiguates an object's path.  This assumes that there is a definite---but uncertain---hidden state that drives the behavior of a target through a fixed dynamical law.  When objects are too small to have texture and can make independent (possibly evasive) actions, we argue that \emph{what is crucial is not the likelihood of a target path, but rather whether it is physically possible and how nearby target paths compete for detections.}

Our approach is to supply hard, but loose, physical constraints to derive a discrete optimization problem, rather than tighter soft probabilistic constraints.  We prove performance guarantees given these constraints, and derive sampling requirements for correct performance.  We will permit objects to lie in an abstract topological target configuration space, rather than a Euclidean space.  We demonstrate an algorithm based on this theory using two feature-poor scenarios:
\begin{enumerate}
  \item 3d multi-target air traffic tracking with multiple sensors that have overlapping domains and varying altitude resolution, and a
  \item Single-sensor range-only sonar.
\end{enumerate}

\section{Outline of the paper}

Because the practical performance of our algorithm depends strongly on the underlying theory, we begin in Section \ref{sec:constraint_functions} with a careful discussion of the general theory of constraint functions.
In Section \ref{sec:sample_rate}, we prove estimates of performance using our approach under the assumption of a probabilistic behavioral model for objects.  It should be emphasized that the probabilistic model is not essential to construct constraint functions, though in practice such a model is useful for training purposes.
In Section \ref{sec:process}, we outline the practical process for using our method.
In Section \ref{sec:results}, we discuss results on our two candidate datasets.
Finally, we conclude in Section \ref{sec:conclusions} with some final remarks.

Our motivation for tracking multiple objects is driven by a need for offline, batch processing.
As such, we take a two-stage, data-driven approach: (1) train the model of behavioral constraints and then (2) run the tracking algorithm using the modeled constraints.
For training, we start with a the observable state space and estimate a \emph{transition probability function} from the data (Definition \ref{def:transition_probability}).
By thresholding this transition probability function, we learn a \emph{constraint function} globally (theoretically: Definition \ref{df:constraint_function}; practically: Section \ref{sec:designing_constraint}).
This learning is run as a constrained optimization over certain classes of functions.
For instance, for kinematic motion, it is sensible to optimize over functions supported on cones emanating from each point in the state space with an axis centered on the velocity vector.
We additionally estimate the amount of variability associated to a given detection.

Once we have estimated the constraint function, we deploy it to construct a \emph{tracklet poset} (Definition \ref{df:tracklet_poset}).
Candidate tracks for objects within our data correspond to maximal chains within the tracklet poset (Corollary \ref{cor:time_chains} and Proposition \ref{prop:search_region}).
We select among those candidate tracks using a weight function that is designed based upon the expected kinematics (Algorithm \ref{alg:posettrack} and Section \ref{sec:weighting}).

\section{Constraints and the tracklet poset}
\label{sec:constraint_functions}

Rather than beginning with a notion of targets with definite locations (usually in a Euclidean space) at definite times, we will represent the set of possible \emph{events} to be observed as the set of points in a topological space $E$.  Given that an event occurs---perhaps a target is at a given location and time---later events may be anticipated.  We encode this inference process on events and observations by the use of a particular kind of \emph{set-valued function}.

The set valued function we have in mind is one that takes events to observations, namely a function from a topological space $E$ to the power set $2^E$.

\begin{definition}
  \label{df:constraint_function}
A set-valued function $C:E \to 2^E$ is a \emph{constraint function} if $x \in C(x)$ for all $x\in E$ and it is a \emph{lower semicontinuous set-valued map}: for every open set $U \subseteq E$, the set $\{x \in E : C(x) \cap U \not= \emptyset\}$ is open.  We often will write 
\begin{equation*}
C(U) = \bigcup_{x \in U} C(x)
\end{equation*}
where $U$ is an open set.   We will often call open sets in $E$ \emph{observations} to emphasize their role in the algorithms that follow.
\end{definition}

A constraint function $C$ represents a set-valued discrete-time dynamical system, relating sets of possible points through a particular dynamical relationship.  Given that an event $x\in E$ has occurred, all possible subsequent events lie in $C(x)$.  We caution the reader that although the intuition is that $C(x)$ consists of events occurring no earlier than $x$, this is not required.  A constraint function can be ``forensic'' and limit possible past events as well.

In the case of tracking, each event $x \in E$ (corresponding to a singleton set $\{x\}$) corresponds to a detection,
whereas an \emph{obsercation} is an open set.
In practice, it is not necessary to belabor this difference, since detections can be interpreted as specifying the centroid of an ellipsoidal open set, whose dimensions are given by the typical measurement errors in each state variable.
We will model this by saying that the true value of the object being tracked lies somewhere within an open neighborhood $U$ that contains $x$.
The constraint function means that the rest of that object's future history will be contained in the open set $C(U)$.

A constraint function $C$ defines a reflexive relation $\prec$ on subsets of $E$: $A \prec B$ if for every open $U$ containing $A$, it follows that $C(U) \cap B \not= \emptyset$.

\begin{definition}
  \label{df:tracklet_poset}
Given a collection $\col{U}$ of subsets of $E$ and a constraint function $C$, the \emph{constraint graph} $TG_{C}(\col{U})$ is the directed graph\footnote{This graph might have cycles if $C$ applies constraints backwards in time.} whose vertices are elements of $\col{U}$ and whose edges are given by those ordered pairs $(A,B)$ for which $A \prec B$.  The \emph{tracklet poset} $TP_{C}(\col{U})$ is given by the equivalence classes of the transitive closure of $TG_{C}(\col{U})$.
\end{definition}

\begin{remark}
  \label{rem:transitive}
  The transitive closure of $\prec$ is typically a preorder: if $A$ and $B$ are distinct overlapping sets, then $A \prec B$ and $B \prec A$ even though $A \not= B$.  Tracklet posets are most useful when built from collections $\col{U}$ of disjoint open sets.  Overlapping open sets in $\col{U}$ will automatically be placed in the same equivalence class.  For truly hard observations, this is the best that one can do.  If there are not hard bounds on the uncertainties of observations, one can substitute smaller, disjoint observations via Proposition \ref{prop:refinement}.  Although it is sometimes useful to consider a tracklet preorder\footnote{A transitive, but not necessarily antisymmetric relation}, it is harder to guarantee correct performance.  Specifically, Proposition \ref{prop:single_target}, Corollary \ref{cor:time_chains}, and Proposition \ref{prop:search_region} all fail to hold, because it is essentially impossible to treat an observation as a single, atomic event.
\end{remark}

\begin{definition}
  \label{def:event_space}
The \emph{event space} generated by a constraint function $C$ on a topological space $E$ is the partial order given by $(TP_{C}(E),\prec)$---elements of $E$ are to be interpreted as singleton sets---which we will usually write $(E,\le)$.
\end{definition}

We will always think of the $\le$ relation as being a temporal order, in which $x \le y$ implies that $x$ comes before $y$.
But the $\le$ relation implies more than that, namely that a single target appearing at $y$ might be the same target as one appearing at $x$.

\begin{remark}
It is often the case that the transitive closure of $\prec$ on singletons will automatically be a partial order, so we will usually abuse notation and treat elements of $E$ as being singletons rather than $\prec$-equivalence classes of singletons.
\end{remark}

In an event space $(E,\le)$ associated to a constraint function $C$ sets of the form
\begin{equation*}
\uparrow x = \{y \in E : x \le y\}
\end{equation*}
will be taken to represent the set of events that could logically follow $x$.  Similarly, we define the \emph{upset} of an observation $U$ to be
\begin{equation}
\uparrow U = \bigcup_{x \in U} \uparrow x.
\end{equation}

\begin{remark}
Although $\uparrow x$ is a closed set in most of the examples in this article, this is not guaranteed by the semicontinuity of the constraint function.
\end{remark}

\begin{example}
  \label{eg:position_space}
$E=X\times \mathbb{R}$ whose components are be interpreted as position and time forms an event space in which $(x,s) \le (y,t)$ if $s \le t$.  $X$ could represent position, position and momentum, or other phase spaces for a dynamical system.  (This idea is taken farther in Definition \ref{df:temporal_event_space}.)
\end{example}

\begin{proposition}
\label{prop:refinement}
If $\col{U}$ is a refinement of $\col{V}$ (each $U \in \col{U}$ is a subset of some $V \in \col{V}$) then this induces an order preserving function $TP_C(\col{U}) \to TP_C(\col{V})$.
\end{proposition}
This induced function is a coarsening from more precise observations to less precise observations.  As noted in Remark \ref{rem:transitive}, this Proposition allows tracklet posets to be refined until observations are disjoint.
\begin{proof}
First, since $\col{U}$ is a refinement of $\col{V}$, there is a function $\phi: \col{U} \to \col{V}$ for which $A \subseteq \phi(A)$ for all $A\in\col{U}$.  What needs to be shown is that $\phi$ is well-defined and order preserving on tracklet posets.  That $\phi$ is well-defined is an immediate consequence of it being order preserving, so it suffices to show that $\phi$ is order preserving.  Suppose that $A \prec B$ for two sets in $\col{U}$.  This means that $B \cap C(A) \not= \emptyset$.  But then $B \subseteq \phi(B)$ and $A \subseteq \phi(A)$.  Since $C$ is a constraint function, the latter statement implies that $C(A) \subseteq C(\phi(A))$.  Thus $\phi(B) \cap C(\phi(A)) \not= \emptyset$, so $\phi(A) \prec \phi(B)$.
\end{proof}

Intuitively, a continuous function from an interval $I$ to $(E,\le)$ might not respect the $\le$ ordering.
Such a function cannot represent the trajectory of a physical object.
We need to distinguish between \emph{paths} (arbitrary continuous functions) and \emph{timelines} (those continuous functions that also respect the ordering).

\begin{definition}
A \emph{timeline} is an order preserving path $p:I \to (E,\le)$ from an interval $I\subseteq \mathbb{R}$ with the usual order. An \emph{observation of $p$} is an observation $U$ that intersects the image of $p$.
\end{definition}

\begin{example}
  \label{eg:velocity_constraint}
  Consider the event space $E=\mathbb{R}^2$ whose components consist of position and time.  If $V>0$ represents a maximum velocity, then the partial order can represent a \emph{forward velocity constraint} whose upsets are generated by sets of the form
  \begin{equation*}
    \uparrow (y,s) = \left\{(x,t) : 0\le \frac{|x-y|}{t-s} < V\right\}.
  \end{equation*}

The timelines for this event space consist of graphs of Lipschitz continuous functions $f:\mathbb{R}\to\mathbb{R}$ whose slope is between $-V$ and $V$.
\end{example}

\begin{example}
  If the event space contains position and velocity, partial orders can enforce the kinematic relationship between them.  For simplicity, consider $E=\mathbb{R}^3$ whose components represent position, velocity, and time.  Suppose that $Q>0$ is given, which represents maximum acceleration.  Define a partial order for which
  \begin{eqnarray*}
    \uparrow (y,w,s) &=&\{(x,v,t) : \text{if there is a path }p:[s,t]\to\mathbb{R}\\
    &&\text{ with } (p(t),p'(t),t)=(x,v,t), \\
    && (p(s),p'(s),s)=(y,w,s) \\
    &&\text{and }|p''(\tau)|< Q \text{ for all }\tau \in [s,t]\}.
  \end{eqnarray*}
  The set of timelines for this event space consists of paths with acceleration less than $Q$. 
\end{example}

According to Proposition \ref{prop:refinement} if $p$ is a timeline $\mathbb{R} \to E$, the composition of $p$ with the induced map $E \to TP_{C}(\col{U})$ is order preserving.  This leads to the following Corollary if the timeline is covered by a refinement of the observations.

\begin{corollary}
If $\col{U}$ is a cover of the image of a timeline (the set of values in the event space), then that timeline corresponds to a chain in $TP_{C}(\col{U})$.
\end{corollary}

This corollary is a bit stringent.
Typically we will not be able to cover the image of a timeline with observations.  The next proposition indicates that this can be relaxed somewhat.

\begin{proposition}
  \label{prop:single_target}
Suppose that $\col{U}$ is a collection of observations in an event space $E$ for a constraint function $C$, that $p:\mathbb{R}^+\to E$ is a timeline for $(E,\le)$ and that $p(0)$ is contained in some $U\in\col{U}$.  Consider the function $P:\mathbb{R}^+ \to TP_{C}(\col{U})$ given by 
\begin{equation*}
P(t) = \begin{cases}
[U(t)]& \text{if there is a } U(t) \in \col{U}\text{ such that }p(t) \in U(t)\\
[U(t')]& \text{where }t'\text{ is the largest number less than }t\\&\text{ such that } p(t')\in U(t') \in \col{U},\\
\end{cases}
\end{equation*}
where we use the square brackets to refer to equivalence classes of observations.
If $P(s) \prec P(t)$, then $s \le t$.
\end{proposition}
$P(t)$ specifies the most recent equivalence class of observations of the timeline up to time $t$.
\begin{proof}
$P$ is well defined because if $U'(t) \in \col{U}$ also contains $p(t)$, then $C(U(t)) \cap U'(t)$ and $C(U'(t))\cap U(t)$ both contain at least $U(t) \cap U'(t)$ by the definition of a constraint function.  But the intersection $U(t) \cap U'(t)$ is nonempty because it contains $p(t)$.  Hence $U \prec U'$ and $U' \prec U$, so we conclude $[U]=[U']$.

Now suppose $P(s) \prec P(t)$.  Since $p$ is a timeline, we have that for any open set $U(s) \in P(s)$, the set $\uparrow U(s)$ contains $p(t)$ and therefore intersects any open set containing $p(t)$.  Thus $\uparrow p(t) \subseteq \uparrow p(s)$, which implies that $s \le t$.
\end{proof}

\begin{figure}
  \begin{center}
    \includegraphics[width=2.75in]{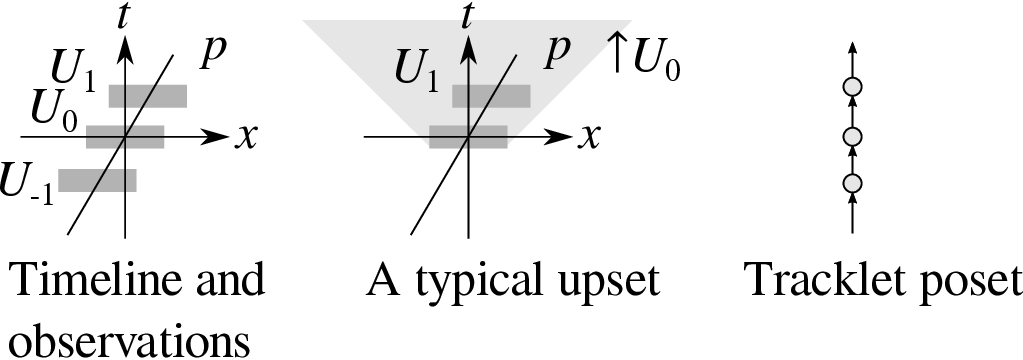}
    \caption{The timeline and some upsets for Example \ref{eg:single_target}}
    \label{fig:single_target}
  \end{center}
\end{figure}

\begin{example}
  \label{eg:single_target}
  Consider the event space in Example \ref{eg:velocity_constraint} with $V=1$ and the timeline $p(t) = (t/2,t)$ (note that the timeline's velocity is $1/2$).  For this timeline, consider the following sequence of observations
  \begin{equation*}
    U_k = \{(x,t) : t/2-1 < x < t/2+1 \text{ and } k-0.1 < t < k+0.1\},
  \end{equation*}
  with the constraint function $C$ given by $C(x) = \uparrow x$.
  The scenario is shown in Figure \ref{fig:single_target}.
  In this case, the tracklet poset $TP_{C}(\{U_k\})$ is isomorphic to the linear ordering of the integers---exactly as the Corollary says.  The function $P$ given by Proposition \ref{prop:single_target} is (up to isomorphism)
  \begin{equation*}
    P(t) = [U_{\lfloor t \rfloor}].
  \end{equation*}
\end{example}

Suppose that $\{p_k:\mathbb{R}^+\to E\}_{k=1}^N$ is a collection of timelines on an event space $E$ for a constraint function $C$.  Suppose that $\col{U}$ is a collection of observations such that for each $U \in \col{U}$ there is a $p_k$ whose image intersects $U$.  Each timeline corresponds to a set of chains in $TP_{C}(\col{U})$, though this correspondence may not be bijective.  Inverting this correspondence---assigning timelines to chains of $TP_{C}(\col{U})$---is the \emph{global track association} problem.

This problem is unsolvable at this level of generality---ambiguities are unavoidable if observations are too limited  (see Examples \ref{eg:two_targets}, \ref{eg:nonspecific}, and \ref{eg:missed_detections})---but a loose bound on the number of targets can be obtained using Dilworth's theorem.  (A tighter bound is obtained in Proposition \ref{prop:timeline_separation}.)

\begin{theorem} (Dilworth's theorem \cite{Dilworth_1950})
For a partial order $P$, there exists an antichain $A$ and a partition of $P$ into a family $K$ of chains, such that the number of chains in $K$ equals the cardinality of $A$.
\end{theorem}

\begin{corollary}
  \label{cor:time_chains}
Suppose that $\col{U}$ is a collection of observations of a set of timelines $\{p_k\}$: each $U\in\col{U}$ intersects the image of at least one timeline and $\col{U}$ covers the union of the images of the timelines.  If the number of timelines is equal to the size of a maximal antichain $A$ in the tracklet poset $TP_{C}(\col{U})$, then one of the families of maximal chains---which partitions $TP_{C}(\col{U})$ (and corresponds to $A$)---corresponds to the set of timelines.  This correspondence is given by a set of functions $P_k:\mathbb{R}^+\to TP_{C}(\col{U})$ as defined in Proposition \ref{prop:single_target}, one for each timeline.  
\end{corollary}

\begin{definition}
  \label{df:search_region}
  The \emph{search region} associated to an event space $(E,\le)$, a collection of observations $\col{U}=\{U_1,\dotsc\}$, and a partition of $TP_{C}(\col{U})$ into chains $c=\{c_1,\dotsc\}$ is given by
  \begin{equation*}
    {\bf S}(c)=\bigcup_i \left( \bigcap_{k=1} \left(\uparrow U_{c_{i,k}}\right) \right),
  \end{equation*}
  where $c_{i,1}, c_{i,2}, \dotsc$ is the list of indices of observations (in $\col{U}$) in the chain $c_i$.  It is of course the case that ${\bf S}(c)$ can also be defined using a maximal antichain of $TP_{C}(\col{U})$.
\end{definition}

\begin{proposition}
  \label{prop:search_region}
  Suppose that $\{p_k:\mathbb{R}\to E\}_{i=1}^N$ is a set of timelines on an event space $(E,\le)$ for a constraint function $C$ and that
\begin{enumerate}
\item $\col{U}$ is a collection of observations of $\{p_k\}$,
\item The size of a maximal antichain $A$ in the tracklet poset $TP_{C}(\col{U})$ is $N$ (the number of timelines), and 
\item All chains of $TP_{C}(\col{U})$ are finite.
\end{enumerate}
Then the search region associated to $A$ contains the image $p_k((T,\infty))$ of each timeline for some sufficiently large finite $T$.
\end{proposition}
\begin{proof}
  By the previous Corollary, one of the maximal antichains corresponds to the set of timelines.  This correspondence assigns a chain of observations $\{\dotsc,U_0,U_1,\dotsc\}$ to each timeline $p$; and that each timeline passes through each such observation in the chain.  The timeline $p$ is order-preserving along this chain of observations by Proposition \ref{prop:single_target}, but since each chain is finite in $TP_{C}(\col{U})$, the function $P$ defined for this timeline in Proposition \ref{prop:single_target} attains its maximum value for some $T<\infty$.  Thereafter, it must lie within the intersection $\bigcap_k \uparrow U_k$.  Taking the (finite) union over all timelines and the (finite) maximum over all such $T$ completes the proof.
\end{proof}

Good \emph{track custody} consists of optimizing against two conflicting requirements:
\begin{enumerate}
\item Constructing an accurate constraint function---so that target trajectories are indeed timelines in the event space, and
\item Ensuring that search regions don't get too big.
\end{enumerate}
Although a probabilistic model of target behavior is not used by the tracklet poset construction,
it is useful to use such a model to characterize its performance in realistic settings.
Under a probabilistic model of target behavior,
the first requirement will tend to require that the event space produce larger sets $\uparrow U$ given an observation $U$ to capture infrequent behaviors or timelines, while the latter requires that that these same sets be smaller.  We make these requirements considerably more precise in Section \ref{sec:sample_rate}.

If only the event space and a collection of observations are available, the best one can do is to find minimal covers of the tracklet poset by chains (namely, find maximal antichains).  There are existing greedy algorithms to do this, for instance a good, but probably non-optimal example is given in \cite{Jourdan_1994}.  The algorithm described in that paper inspires Algorithm \ref{alg:posettrack} described in the later sections of this article.
 
Nevertheless, finding a minimal cover by maximal chains is the best that can be done at this level of generality.  Problems still can arise, as the following examples show.  Ultimately, we will distinguish between chains and \emph{tracks}.  Tracks are chains that have been selected by our Algorithm as being the actual tajectories of objects.

\begin{figure}
  \begin{center}
    \includegraphics[width=2.75in]{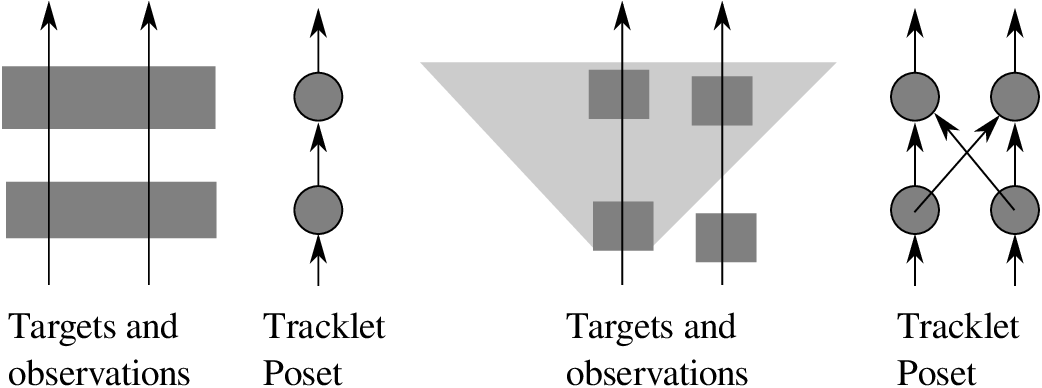}
    \caption{Two timelines, observations of them, and their associated tracklet posets. At left, there is aliasing because observations touch both timelines.  At right, the observations separate the two timelines, but two different sets of maximal chains exist.}
    \label{fig:two_targets}
  \end{center}
\end{figure}

\begin{example}
  \label{eg:two_targets}
Given a single antichain, there may be multiple different sets of maximal chains.  This quite often happens if two timelines are far enough apart to be separated by individual observations, but are not kinematically far apart to remain separable.  This situation is shown at right in Figure \ref{fig:two_targets}.
\end{example}

If the hypothesis given by the Corollary is violated, more trouble can occur.

\begin{example}
  \label{eg:nonspecific}
For instance, as shown at left in Figure \ref{fig:two_targets}, two timelines can be covered by non-specific observations.  This results in the two timelines being aliased into a single maximal chain.
\end{example}

False observations can result in spurious maximal chains.

\begin{figure}
  \begin{center}
    \includegraphics[width=1.25in]{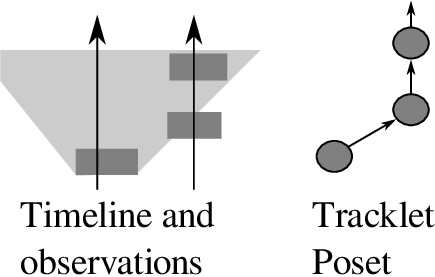}
    \caption{Too sparse observations result in misleading tracklet posets}
    \label{fig:missed_detections}
  \end{center}
\end{figure}

\begin{example} 
  \label{eg:missed_detections}
If the observations are too sparse, the chains can also become misleading.
For instance, Figure \ref{fig:missed_detections} shows an example in which there is one initial observation of the timeline and then many other false---but plausible---observations from other timelines.  We can cover this tracklet poset with one chain, even though there are multiple timelines.  The conclusion is that we must obtain enough observations even to count the number of targets present.  This is studied more extensively in Section \ref{sec:sample_rate}.
\end{example}

\section{Controlling search region size}
\label{sec:sample_rate}

When the event space is fixed by the problem---and possibly known \emph{a priori}---control of the size of search regions is essentially a matter of keeping observations small enough.  We divide the event space into a space and time portion to focus on the problem of tracking objects.

\begin{definition}
  \label{df:temporal_event_space}
A \emph{temporal event space} is an event space of the form $(E,\le)=(X\times\mathbb{R},\le)$ in which $(x,s) \le (y,t)$ implies $s \le t$.  (The reader is cautioned that the partial orders $\le$ and $\prec$ on $E$ still represent \emph{causality} between events and observations.)  We call the first component (in $X$) \emph{space} and the second coordinate \emph{time}.
\end{definition}

\subsection{Target separation}

We first address the possibility of confusing targets as described in Examples \ref{eg:two_targets} and \ref{eg:nonspecific} in the previous section.

\begin{proposition}
  \label{prop:timeline_separation}
  Assume that $(X,d)$ is a metric space and that $\col{P}=\{p_k:\mathbb{R}\to X\}$ is a collection of functions whose graphs\footnote{The \emph{graph} of the timeline as a function.} are timelines for a constraint function $C$ that generates a temporal event space $(X\times\mathbb{R},\le)$ and whose points are always separated by some minimal distance $\delta$
  \begin{equation*}
    d(p_i(t),p_j(t)) > \delta\text{ for all }t.
  \end{equation*}

  If $\col{U}$ is a collection of observations for which the diameter of the set $C(U) \cap X \times\{t\}$ (measured in the metric $d$) is less than $\delta/2$ for each $U\in\col{U}$ and $t\in\mathbb{R}$, then there is only one partition of $TP_{C}(\col{U})$ into maximal chains.
\end{proposition}
The hypotheses of the Proposition imply that the only ambiguity in the correspondence between maximal chains of $TP_{C}(\col{U})$ and the set of timelines is that any given timeline may be broken up into several maximal chains.  If the constraint function is given by $C(x) = \uparrow x$, it will usually be impossible to satisfy the hypotheses because the diameter of $\uparrow U$ is typically infinite.

\begin{proof}
  The hypotheses imply that if $U_1$ and $U_2$ are two observations that intersect two distinct timelines $p_1$ and $p_2$ respectively, then $C(U_1) \cap U_2 = \emptyset$ and $C(U_2) \cap U_1 = \emptyset$.  Thus $U_1 \nprec U_2$ and $U_2 \nprec U_1$.
\end{proof}

To avoid searching for timelines over too much area, we can quantify and bound the area contained in the search region.

\begin{proposition}
  \label{prop:search_region_size}
  Suppose that $\col{P}=\{p_k:\mathbb{R} \to (X\times\mathbb{R},\le)\}$ is a collection of $N$ timelines in a temporal event space $(X\times \mathbb{R},\le)$ and that $m$ is a measure on $X$.
  Suppose that $\col{U}$ is a collection of observations of $\col{P}$ satisfying 
\begin{equation}
\label{eq:measure_bound}
m((\uparrow U) \cap X\times\{t\}) < M(t)
\end{equation}
for each $U\in\col{U}$ and each $t\in\mathbb{R}$, where $M$ is an increasing function of $t$.
  Then
  \begin{equation*}
    m({\bf S}(c)\cap X\times\{t\}) \le M(t) N < \infty
  \end{equation*}
  for all $t\in\mathbb{R}$.
\end{proposition}

\begin{proof}
  By Definition \ref{df:search_region} and Corollary \ref{cor:time_chains}, we have that
  \begin{equation*}
    {\bf S}(c)=\bigcup_i \left( \bigcap_{k=1} \left(\uparrow U_{c_{i,k}}\right) \right),
  \end{equation*}
  where $c$ is the set of chains corresponding to the finite collection of timelines.  This means that
  \begin{eqnarray*}
    m\left({\bf S}(c) \cap X\times\{t\}\right) &\le& \sum_{i=1}^N m\left( X\times\{t\} \cap \bigcap_{k=1} \left(\uparrow U_{c_{i,k}}\right) \right)\\
    &\le&\sum_{i=1}^N m\left( X\times\{t\} \cap (\uparrow U_{c_{i,1}})\right) \\
    &\le& M(t) N,
  \end{eqnarray*}
  as desired.
\end{proof}

\subsection{Maintaining custody of timelines}

If a probabilistic model of target behavior is available, it is reasonable to use this model to assess how well the constraint function agrees with it.
This will allow us to assess performance of any tracking algorithm based upon the constraint function.
The probability that a timeline is consistent with a constraint function quantifies how well one can keep custody of the timelines.  There is a trade off between the size of the search region and the required frequency of observations required to maintain this probability.

A continuous path $p:\mathbb{R} \to E$ in an event space might not be consistent with a given constraint function $C$.  If $p$ represents a true trajectory of the system being observed, the tracklet poset associated to a set of observations of $p$ might not correspond to $p$ according to our methodology thus far; this is a \emph{loss of custody}.  Of course, using the upsets of $\le$ would ensure no loss of custody, but this causes ambiguity if there are multiple timelines and the search region can be impractically large.

\begin{definition}
  \label{def:transition_probability}
  Suppose that $(E,\le) = (X \times \mathbb{R},\le)$ is a temporal event space.  The \emph{transition probability density function} $P$ is a Borel measure on $E\times E = (X\times\mathbb{R})\times (X\times\mathbb{R})$ such that
\begin{enumerate}
\item Each time slice $t\in \mathbb{R}$ is a probability density function
\begin{equation*}
P((x,s),(y,t))=P_{(x,s),t}(y)
\end{equation*}
for each $(x,s) \in E$,
\item $P$ respects causality: $P((x,s),(y,t)) = 0$ whenever $(y,t) \le (x,s)$,
\item $P$ is continuous: for each measurable $A \subseteq E$, 
\begin{equation*}
\lim_{t\to t'} |P_{(x,s),t}(A \cap \{t\}) - P_{(x,s),t'}(A\cap \{t'\})| = 0, \text{and}
\end{equation*}
\item $P$ is conditioned on $(x,s)$: $P_{(x,s),s} = \delta_x$.
\end{enumerate}
\end{definition}

In this situation, we can determine the likelihood that timelines can escape from a given constraint function.

\begin{definition}
\label{def:loss_custody}
  Suppose $C$ is a constraint function on the temporal event space $(E,\le)=(X\times\mathbb{R},\le)$ and that $P$ is a transition probability function.  The \emph{loss of custody probability up to time $T$} is given by
  \begin{equation*}
    LC_{C}(T) = 1 - \inf_{(x,s)\in E} \inf_{U\text{ open,} \atop (x,s)\in U} \inf_{s \le t \le s+T} P_{(x,s),t}(C(U) \cap \{t\}).
  \end{equation*}
\end{definition}

In Section \ref{sec:airtraffic}, we show how this loss of custody probability can be estimated from training data that includes identities of the timelines.  By tuning parameters of the constraint function $C$, one tries to ensure that the loss of custody probability is kept to an acceptably low value.  Generally, $LC_{C}(T)$ need not be a monotonic function of $T$.  This especially happens if the timelines exhibit recurrent behaviors, returning to the same (predictable) place multiple times.  

\begin{proposition}
\label{prop:loss_custody}
$LC_{C}(T)$ is an upper semicontinuous function of $T$ for which $LC_{C}(0)=0$.  Thus, given an $L>0$, there is a $T'>0$ for which all $0 \le T < T'$ satisfy
\begin{equation*}
LC_{C}(T) < L.
\end{equation*}
\end{proposition}

This Proposition states that under the hypotheses on $P$, and for a fixed constraint function, it is always possible to reduce the sampling interval to meet some desired loss of custody bound.  Before proving Proposition \ref{prop:loss_custody}, we prove the following useful Lemma.

\begin{lemma}
If $f: X \to \mathbb{R}$ is a lower semicontinuous function on a topological space $X$ and $A$ is an upper semicontinuous set-valued map from $\mathbb{R}$ to the open sets of $X$, then $g:\mathbb{R} \to \mathbb{R}$ given by
\begin{equation*}
g(t) = \inf_{x\in A(t)} f(x)
\end{equation*}
is lower semicontinuous.
\end{lemma}
\begin{proof}
Let $a\in\mathbb{R}$ be given.  We wish to show that $g^{-1}((a,\infty))$ is open, by showing that
\begin{equation*}
  g^{-1}((a,\infty)) = \{t \in \mathbb{R} : A(t) \subseteq f^{-1}((a,\infty)) \},
\end{equation*}
the right side of which is open because $A$ is upper semicontinuous and $f$ is lower semicontinuous.

Suppose that $t \in g^{-1}((a,\infty))$.  This implies that $a < \inf_{x\in A(t)} f(x)$, which is equivalent to the statement that $a < f(x)$ for all $x\in A(t)$, which is itself equivalent to $A(t) \subseteq f^{-1}((a,\infty))$.

Conversely, suppose that $s$ is such that $A(s) \subseteq f^{-1}((a,\infty))$.  This implies that $\inf_{x \in A(s)} f(x) > a$, so that $g(s) > a$, or in other words $s\in g^{-1}((a,\infty))$.
\end{proof}

\begin{proof} (Proposition \ref{prop:loss_custody})
Observe that 
\begin{eqnarray*}
LC_{C}(0) &=& 1 - \inf_{(x,s)\in E} \inf_{U\text{ open,} \atop (x,s)\in U} P_{(x,s),s}(C(U) \cap \{s\})\\
&=& 1 - \inf_{(x,s)\in E} \inf_{U\text{ open,} \atop (x,s)\in U} \delta_x(C(U) \cap \{s\})\\
&=& 0,
\end{eqnarray*}
because $C$ is a constraint function, and so $U\subseteq C(U)$ for all open $U$.

To show that $LC_{C}$ is upper semicontinuous, it suffices to show that
\begin{equation*}
 h(T) = \inf_{s \le t \le s+T} P_{(x,s),t}(C(U) \cap \{t\})  
\end{equation*}
is lower semicontinuous for $T \ge 0$.  This is a direct application of the Lemma as follows.  Let
\begin{equation*}
  f(t) = P_{(x,s),t}(C(U) \cap \{t\}),
\end{equation*}
which is continuous by the continuity of $P$ and the fact that $C(U)$ is Borel measurable.  (The Lemma only requires semicontinuity.)  Also let $A$ be the set-valued function defined by
\begin{equation*}
  A(T) = [s,s+T],
\end{equation*}
which is upper semicontinuous for $T \ge 0$.
\end{proof}

\section{Tracking process flow}
\label{sec:process}

Our approach to tracking proceeds through three stages:
\begin{enumerate}
\item Designing the constraint function based upon training data,
\item Constructing the tracklet poset using the constraint function,
\item Applying a weighted optimization to select \emph{tracks} from possible timelines.
\end{enumerate}
Stages (2) and (3) are a batch algorithm, suited for forensic analysis rather than for realtime tracking.
Should an incremental tracking algorithm be needed, Stage (2) can be constructed incrementally.
It is future work to extend Stage (3) to an incremental algorithm.

The constraint function is designed offline through a manual characterization of the \emph{possible} behaviors of a single target.  Since all of the observations in our dataset are labeled with target identity, it is possible to visualize the set of all possibilities for future time-location pairs, and to extract the necessary parameters from this visualization.

As noted, although the theory of constraint functions distinguishes between observations (open sets) and detections (singletons),
this difference is not practically that important.  By slightly enlarging the open sets generated by the estimated constraint function,
error in treating detections as observations did not impact our results.

Given a set of observations, Definition \ref{df:tracklet_poset} gives an algorithmic prescription for the tracklet poset.  We implemented the tracklet poset as a directed graph whose edges are pairs of observations.  To avoid computational issues, we only constructed edges using pairs of observations that fell within a sliding window of time.

As an aside, if the sample rate present in the data is too low or the constraint function is too lax then the hypotheses of Proposition \ref{prop:timeline_separation} do not apply.  In that case, separate targets will become ambiguous.  As the Proposition states, the possible behaviors of the targets control the necessary sample rate: more maneuverable targets require higher sample rates.  Although we do not usually have access to the actual (continuous) target timelines, with training data it is possible to estimate the required sampling rate to keep target timelines separated if the target identities are known.

\subsection{Local weighting optimization}
\label{sec:weighting}

One way to resolve ambiguities in the track association problem is to prefer certain timelines over others, thereby enriching the model.
We need to determine which possible timelines of an object (a \emph{chain} in the tracklet poset) correspond to a \emph{track} describing an object's actual trajectory.
Our approach is to enrich the tracklet poset graph $TP_{C}(\col{U})$ with weights.  We interpret the weight of an edge as the minimum amount of energy necessary to follow a timeline from one observation to the next along that edge.  Larger weights require more energy, and therefore are less kinematically likely.

Once the tracklet poset graph is obtained, we apply Algorithm \ref{alg:posettrack} and proceeds in two main stages,
which are sub-stages of Stage 3.  In Stage 3A, we add edges in order of least weight to greatest weight, but no edges are added that would cause a vertex to exceed an indegree or outdegree of 1.  Once this is complete, the resulting reduced graph is disconnected, and each connected component is a chain.  Therefore, Stage 3B extracts chains by merely iterating over the set of vertices with indegree 0.

The precise methodology for assigning weights is a driver of overall tracker performance, and needs to respect the underlying dynamics of the individual timelines.  It is necessary to tailor the weighting procedure to the particular timeline behaviors, though our tracking algorithm is only sensitive to the ordering of edges of $TP_{C}(\col{U})$ in terms of their respective weights.

Since weighting is an important performance driver, one could ask if weighting \emph{alone} is sufficient for good tracker performance, in essence testing whether the tracklet poset is necessary at all.  The easiest way to do this is to replace the tracklet poset with a complete graph, so that every detection is reachable from every other.  For the air traffic example, this is simply computationally infeasible, so starting with the chains of the tracklet poset at least provides a significant computational advantage.  For the sonar example it is possible to perform the desired experiment.  In that case (the sonar dataset run with the tracklet poset disabled), the results were so poor as to not be worthy of further note.

We applied two separate weighting methodologies:
\begin{enumerate}
\item A ``simple'' weight, which is given by the spatial distance (horizontal and vertical) between observations, and
\item A ``tailored'' weight, which is the sum of six normalized functions between observations:
\begin{enumerate}
\item Horizontal distance,
\item Vertical distance,
\item Time difference,
\item Heading difference,
\item Speed difference, and
\item a Kinematic difference based on the forward projected states using a Kalman filter. 
\end{enumerate}
\end{enumerate}

Unlike the Ford-Fulkerson algorithm \cite{Ford_1956}, in which edges can be reused up to a finite capacity, our algorithm limits the number of times \emph{vertices} can be reused.  Specifically, each vertex has indegree and outdegree each at most 1.  Recalling that vertices are observations of a timeline in $TP_{C}(\col{U})$, this means that we assume that timelines cannot intersect.   

\begin{algorithm}
\caption{Poset tracking algorithm}
\label{alg:posettrack}
\begin{algorithmic}
\Function{PosetTrack}{$dets$,$E$}
\Require{$dets$ is the list of observations}
\Require{Tracklet poset graph given its list of edges $E\subseteq dets\times dets$}
\Require{The edges in $E$ are listed in order of increasing weight}

\State \Comment{{\bf Stage 3A: Reduce the tracklet poset graph}}
\State Set $F \gets \{\}$ \Comment{List of edges forming unambiguous minimal-weight chains}
\State Initialize a boolean list $srcDets$ of length $dets$ to all False
\State Initialize a boolean list $endDets$ of length $dets$ to all False

\ForAll{$e=(p_1,p_2) \in E$}
 \If{not $srcDets[p_1]$ and not $endDets[p_2]$}
  \State Append $e$ to $F$
  \State $srcDets[p_1] \gets True$
  \State $endDets[p_2] \gets True$
 \EndIf
\EndFor

\State \Comment{{\bf Stage 3B: Extract tracks from the chains of the reduced graph}}
\State Set $tracks \gets \{\}$ \Comment{The set of chains selected as tracks}

\ForAll{$d$ in $dets$}
 \If{$srcDets[d]$ and not $endDets[d]$}
  \Comment{This detection starts a track}
  \State Append $FollowTrack(d,F,\{\})$ to $tracks$
 \EndIf
\EndFor
\State \Return tracks
\EndFunction

\Function{FollowTrack}{$d$,$F$,$track$}
\ForAll{$e=(p_1,p_2)\in F$}
 \If{$p_1==d$}
  \State Set $newtrack$ to be the concatenation of $track$ and $e$
  \State Append $FollowTrack(p_2,F,newtrack)$ to $track$
 \EndIf
\EndFor
\State \Return $track$ appended with $d$
\EndFunction
\end{algorithmic}
\end{algorithm}

\section{Results}
\label{sec:results}

We deployed Algorithm \ref{alg:posettrack} and a tracker based on a standard Kalman filter against two datasets to test the applicability of the sample rate bounds in Section \ref{sec:sample_rate} under reasonable constraints on search regions and to compare the performance of the algorithm against a baseline approach.  We did not tailor the Kalman filter beyond optimal selection of initial covariance matrices, as the goal was merely to provide a meaningful comparison and not a comparison against the state of the art.

Since a single Kalman filter does not associate detections from multiple objects into tracks,
we wrote a simple track associator that connects each detection to its nearest neighbor in state and covariance,
provided one exists within a given fixed window.  

Specifically, we tested the algorithm against 
\begin{enumerate}
\item An commercial air traffic monitoring dataset: in which there is relatively low clutter, few occlusions, but many targets executing complex behaviors.  This dataset also was independently tagged with target identities, which allowed for direct measurement of sample rate bounds.
\item An indoor sonar dataset: in which there is high noise, high clutter, and many occlusions.  Targets consisted of specular points of reflection, which were not independently identified.
\end{enumerate}

In order to baseline our expectations for tracker performance, we compared performance against a tracker based on a standard Kalman-filter that incorporates position and velocity estimates into its state space.  We tuned the observational and initial state covariances of the Kalman filter for optimal performance.  We recognize that this is \emph{not} a state-of-the-art Kalman-based tracker---performance improvements are certainly possible---but merely that it is \emph{well-understood} and therefore a good baseline.

\subsection{Monitoring air traffic}
\label{sec:airtraffic}

Positions and identities of each aircraft in the volume in our dataset were reported through the Automatic Dependent Surveillance-Broadcast (ADS-B) system.  The ADS-B system is a beacon system by which aircraft automatically report their positions at regular intervals using a radio frequency (RF) transmitter.  The ability to maintain track on many aircraft is dependent on their voluntary transmission of this information.  Each aircraft tracks its own position using GPS, and sends information on identification, position in latitude and longitude, altitude, and velocity via data link.  Receivers on the ground capture this information and aggregate it to form a large scale picture of aircraft locations.

ADS-B data has several useful characteristics for testing tracker algorithms.  These include (1) flight paths that are consistent with actual flight operations, (2) realistic update rates that are adequate to maintain track custody, (3) intermittent drop-outs from extended ranges, terrain shadowing, and anomalous propagation, and (4) imperfect knowledge of position for each detection due to the RF propagation environment.  However, (1) is also a limitation, since by using ADS-B data one is limited to the actual flight profiles and flight density of commercial air traffic. 

For this study, ADS-B data were retrieved from the PlanePlotter \cite{PlanePlotter} service.  To improve the data quality from the raw ADS-B detections, those tracks that were especially sparse were manually discarded.  In those cases where two detection articles showed that an aircraft maintained a constant course and speed from one point to the other, without maneuvering, an additional point was added in between.

Sixteen (16) subsets of the original 40 minute dataset were produced.  These were parameterized by a target density and a time subsample factor, so each dataset is identified as ``$m\_n$'' in which $m$ is the density factor and $n$ is the subsample factor, where $m$ and $n$ range from 1 to 4.  The density factor indicates that every $m$-th target from the original dataset is present in the dataset, while the time subsample factor indicates that every $n$-th time sample is present in the dataset.  Thus, the dataset $1\_1$ is the original dataset, while $4\_4$ was the sparsest.

\begin{figure}
  \begin{center}
    \includegraphics[width=2.75in]{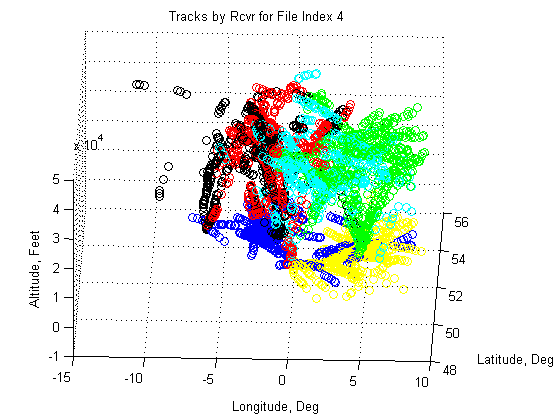}
    \caption{A view of the air traffic data with no subsampling, colored by target identity.}
    \label{fig:airdata_summary}
  \end{center}
\end{figure}

The goal of the tracking algorithms is to produce tracks in which all ADS-B returns (\emph{detections}) from a given target ID are present in a single track, but no detections from any other target ID appear in that track.  Deviations from this ideal situation are quite common, and can be analyzed statistically via several measures of track quality.

In order to test the performance of Algorithm \ref{alg:posettrack} as applied to aircraft tracks, we fed the position and time information into the Algorithm but removed the self-reported aircraft IDs.  The Algorithm produced tracks with replacement IDs, which we compared against the self-reported IDs.

\subsubsection{Designing the constraint function}
\label{sec:designing_constraint}

\begin{figure}
  \begin{center}
    \includegraphics[width=2.75in]{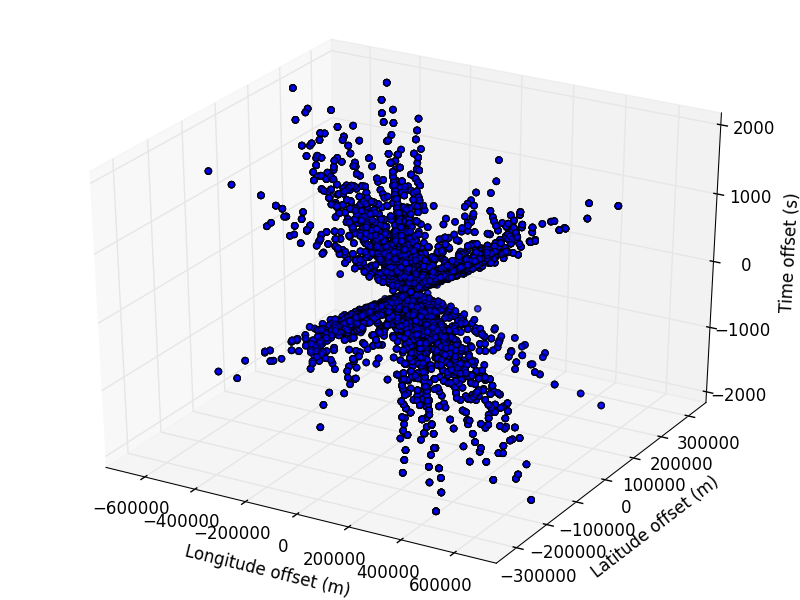}
    \caption{Distribution of previous and subsequent detections centered on a given detection, aggregated over all targets in the density 4, subsample 4 dataset.  The vertical axis of the plot is time offset, while the horizontal axes are latitude and longitude.  Vertical displacement is not shown.}
    \label{fig:airdata_neighborhood}
  \end{center}
\end{figure}

Figure \ref{fig:airdata_neighborhood} shows a visualization we used to construct the constraint function, which consists of a point cloud centered on the purported location and time of a given detection.  Each point shown in the Figure is a future (or past) displacement-time offset pair for a target in our dataset.  Specifically, suppose that $\dotsc, (x_{-1},t_{-1}), (x_{0},t_{0}), (x_{1},t_{1}),\dotsc$ is a sequence of location-time pairs for a given target (all of the same identity).  Then the set of points given by 
\begin{equation}
\label{eq:targets_centered}
\bigcup_{k} \bigcup_{j} (x_j - x_k, t_j - t_k)
\end{equation}
lists all possible target displacement-time offset pairs for that target throughout its timeline.  Figure \ref{fig:airdata_neighborhood} shows this set aggregated over all targets in the lowest density dataset (target density 4 with every one sample used out of each 4 samples taken).  This set should be contained in the $C(U)$ for any open set $U$ containing the origin, which provides a minimum size for $C(U)$.  

Based on visual examination of Figure \ref{fig:airdata_neighborhood}, we realized that target motion is essentially latitude/longitude translation invariant.  The starting latitude/longitude of the target does not appear to bias its motion particularly strongly.  

Although vertical motion (not shown) does stratify into different behaviors---takeoff, landing, cruising---this did not seem to be highly correlated with the latitude/longitude of the target.  There is a clear horizontal speed cutoff, evident in the conical gap in the middle of the Figure.  This corresponds to a cutoff of just over 300 m/s, which is the maximum speed we observed.  Most targets have subsequent samples within 500 km and 300 s of the original sample.  In what follows, we used a constraint function given by 
\begin{eqnarray*}
C(U) &=& \left\{(x,t) \in \mathbb{R}^3\times\mathbb{R} : \text{there is an } (y,s) \in U\text{ such that}\right.\\
&& 0 \le t-s < 300\,\text{s}\\
&& \sqrt{(x_1-y_1)^2+(x_2-y_2)^2}<500\,\text{km},\\
&& |x_3-y_3|<500\,\text{m, and}\\
&&\left.0 \le \frac{\sqrt{(x_1-y_1)^2+(x_2-y_2)^2}}{t-s} < 300\,\text{m/s} \right\}.
\end{eqnarray*}

A measure of the correctness of this choice of constraint function is how well it retains custody of a single target.  According to Definition \ref{def:loss_custody}, \emph{loss of custody} occurs when a detection of a target exits the search region for that target.  For instance, if a helicopter's flight path is present in the data, it would satisfy the velocity constraint.  But only helicopters were used to build $C$, then few of the aircraft would satisfy that constraint.  This can be estimated easily as the percentage of points in the set given in \eqref{eq:targets_centered} that are not contained in $C(U)$ for some open neighborhood $U$ of the origin.  Using $C$ as the constraint function, the mean loss of custody rate is 3.7\% over all targets in the full dataset with all targets (density 1 subsample 1).  If we remove the time gate---permitting $t-s \ge 300$ above---then this value decreases slightly to 2.6\%, which indicates that a small number of targets are sampled extremely infrequently.  As shown in Figure \ref{fig:losscustody}, the loss of custody is fairly stable over all datasets, so we conclude that the constraint function defined above is correct.

\begin{figure}
  \begin{center}
    \includegraphics[width=2.5in]{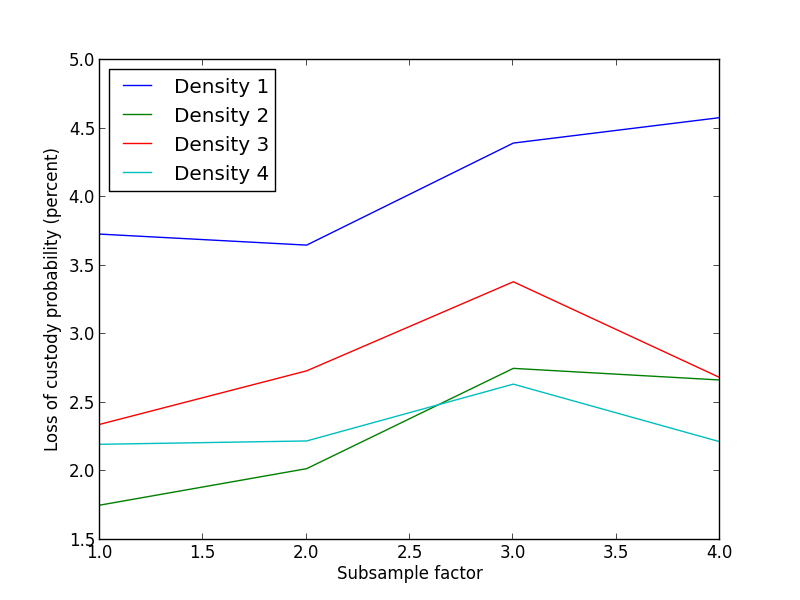}
    \caption{Loss of custody for the datasets as a function of subsample rate}
    \label{fig:losscustody}
  \end{center}
\end{figure}

\subsubsection{Sample rates for disambiguation}

\begin{figure}
  \begin{center}
    \includegraphics[width=2.5in]{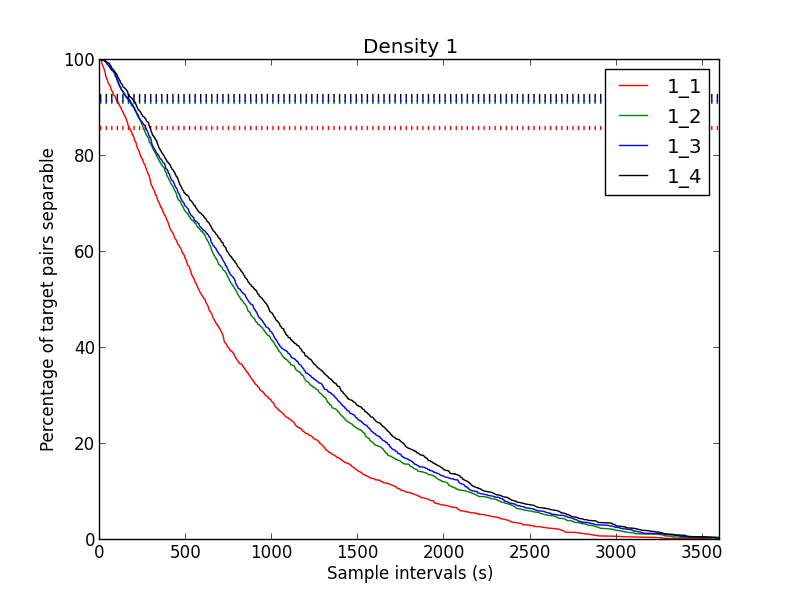}
    \caption{Percentage of target pairs that can be separated at a given sample interval.  The horizontal axis specifies the sampling interval (seconds), while the vertical axis specifies the percentage of target pairs that can be separated at that sampling interval.  The horizontal dotted lines indicate the actual percentage of targets that are separable in the data.}
    \label{fig:samplerates}
  \end{center}
\end{figure}

Proposition \ref{prop:timeline_separation} requires estimation of the diameter of $C(U) \cap X \times\{t\}$ for the constraint function $C$ as a function of time.  Since the Proposition requires the minimum distance between timelines---which we do not know, since the data are already sampled---we interpolated linearly between samples.  Using the constraint function $C$, this diameter is bounded according to the maximum target speed: 300 m/s.  Figure \ref{fig:samplerates} shows that approximately 90\% of target pairs are separable using a maximum speed of 300 m/s.  This is indicated by the horizontal dotted lines in the Figure.  This corresponds closely to a sample interval of approximately 300 s.  Therefore, all trackers in our experiment were instructed to drop tracks that have gone stale after 300 s.

\subsubsection{Statistical performance overview}

To compare tracker performance in a statistically-sound and algorithm-agnostic way across all three tracker algorithms (Kalman tracker and Algorithm \ref{alg:posettrack} with two different weighting strategies), we focused first on the overall quality of the tracks produced.  Since the identities of each detection were known, we can produce two measures of track quality:
\begin{enumerate}
\item The number of tracks associated to a given target identity (``tracks per target'', a measure of how frequently tracks are dropped) and
\item The number of different target identities represented in a given track (``targets per track'', a measure of track identity uncertainty). 
\end{enumerate}
Ideally, both measures should be as close to 1 as possible.  We note that the linear combination in the tailored weight can be adjusted to get a balance between targets per track and tracks per target performance and could easily be adjusted to improve each separately at the expense of the other.

\begin{figure*}
  \begin{center}
    \includegraphics[width=5in]{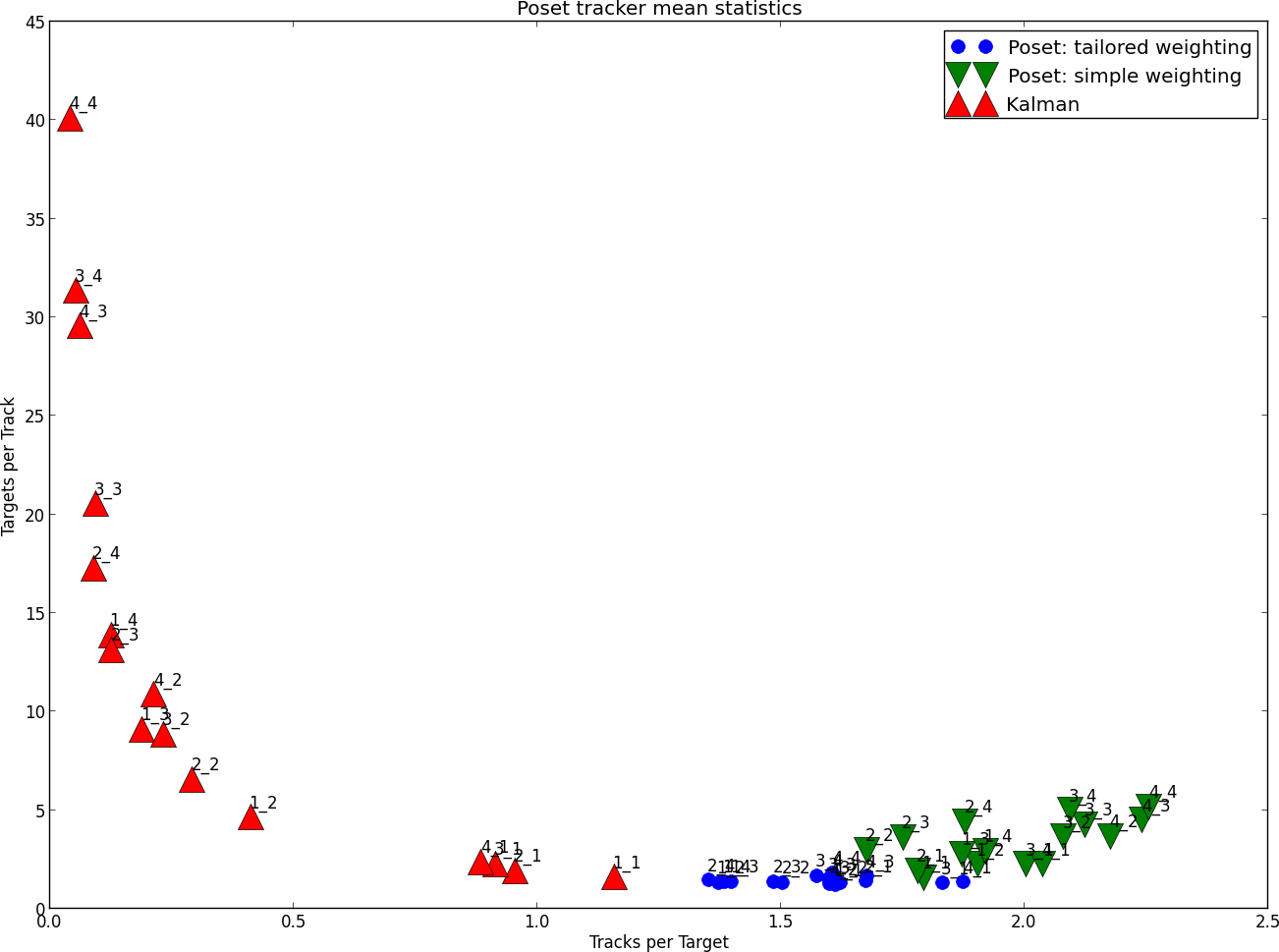}
    \caption{Tracker performance for different datasets showing targets per track versus tracks per target.  Each marker is labeled with an ordered pair $A$\_$B$ in which $A$ is the density of the tracks and $B$ is the subsampling factor.  Red triangles correspond to the Kalman tracker, green triangles correspond to the poset tracker with simple neighborhoods, and blue circles correspond to the poset tracker with refined neighborhoods.}
    \label{fig:overall_plots}
  \end{center}
\end{figure*}

As shown in Figure \ref{fig:overall_plots}, both weighting methodologies yield an improvement over a standard Kalman filter tracker, though the tailored weight yields a much more substantial improvement.  The general behavior of track quality as the density and sampling rates change are as we would expect, with the ``Density 1 Sample 1'' (i.e. ``1\_1'') dataset having better performance and the ``4\_4'' dataset the worst performance.  There is a trend of degrading performance as the densities increase, a rough monotonicity to ``best'' balanced performance from lower to higher sampling rates (i.e. ``1\_1'' to ``1\_4'').  In the poset trackers, this shifts from better targets per track (ambiguity) to better tracks per target (continuity) as the sampling rate gets higher. 

Figure \ref{fig:overall_plots} shows that the Kalman tracker has difficulty initiating tracks since its typical number of tracks per target is below 1.  Worse, when the Kalman tracker initiates a track it tends to mix identities, especially in the case of undersampled, high density targets.  The poset-based trackers tend to result in substantially more disconnected, tracks with unambiguous identities.  The next two subsections examine target length and ambiguity more closely, since Figure \ref{fig:overall_plots} shows \emph{mean} performance only.  A distributional analysis is helpful in clarifying these performance estimates.

\subsubsection{Track length analysis}

The ``tracks per target'' measure is fairly coarse---a tracker that produces one long track with a few short ones for a single target is indistinguishable from one that produces the same number of medium-length tracks.  An easy way to rectify this is to examine the typical length of a track, measured against the actual length of the target's track.  Again, since we have the identities of the targets, it is easy to compute the ratio of the length of a track versus the number of observations.  This is shown in Figure \ref{fig:lenratio} for the three trackers under discussion.  (To reduce clutter in the plot, we used only the longest track for each target.)  

\begin{figure*}
  \begin{center}
    \includegraphics[width=6in]{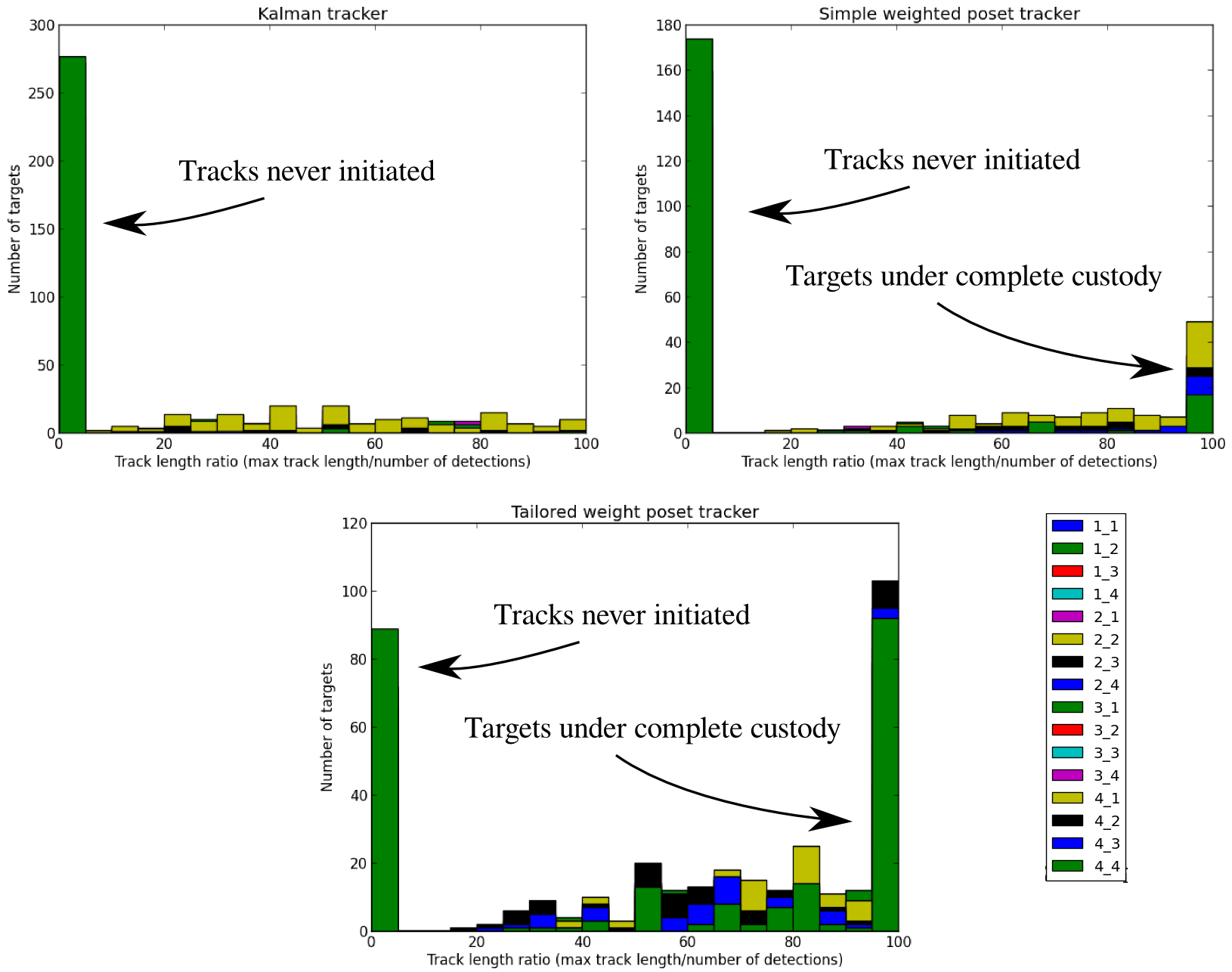}
    \caption{Histograms of ratios of track length: estimated versus actual; colors indicate different datasets}
    \label{fig:lenratio}
  \end{center}
\end{figure*}

The leftmost bar represents those tracks that are only one detection long---those tracks that were never really followed.  Each tracker we studied fails to initiate some tracks, though both poset trackers do a substantially better job of initiating tracks than the Kalman filter.  The rightmost bar represents tracks that cover the entire extent of the observations for a given target---targets for which custody was never lost.  The Kalman filter rarely produces tracks of this sort, while both poset trackers were able to retain custody on a large proportion of the targets.  

The tracker using the tailored weight permits about twice as many tracks to be initiated and maintains complete custody on approximately twice as many targets over the simple weight according to the right and bottom frames of Figure \ref{fig:lenratio}.  In addition to these clear-cut cases, it also picks up more challenging targets for which it cannot maintain full custody, as evidenced by the increase over the middle of the plots.

\subsubsection{Misidentification analysis}

Similar to the track length distributions shown above, it is possible to refine the ``targets per track'' measure by analyzing the proportion of target identities that are incorrect in a given track.    

\begin{figure*}
  \begin{center}
    \includegraphics[width=6in]{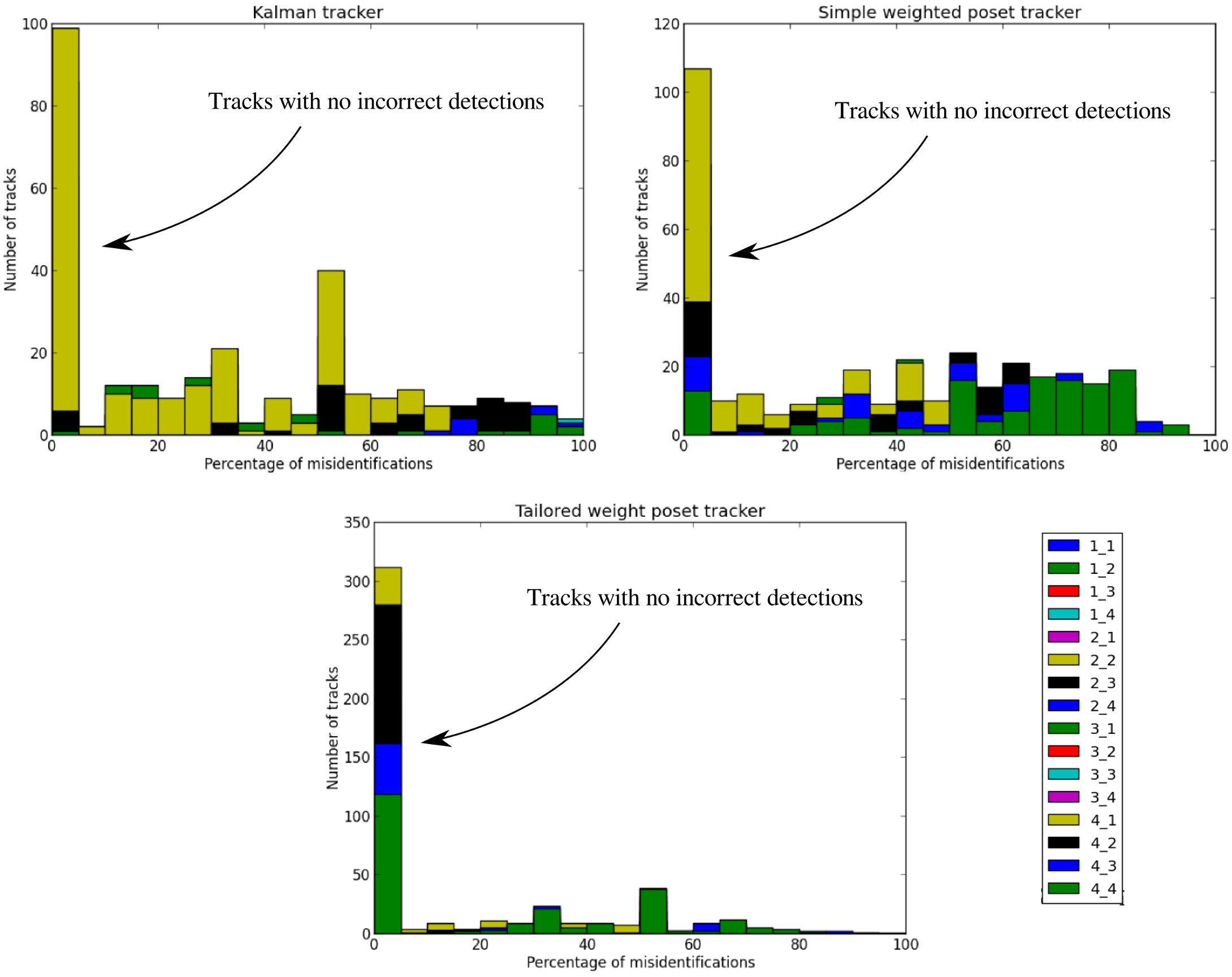}
    \caption{Histogram of proportion of incorrect target identities in a given track; colors indicate different datasets}
    \label{fig:misid}
  \end{center}
\end{figure*}

Figure \ref{fig:misid} shows the performance of all three trackers in this case.  A given track tends to be generally more ambiguous with the Kalman tracker than the poset tracker with tailored weight, but this is highly dependent on having a good weighting.  In the Figure, the leftmost bar consists of unambiguous tracks, those that contain only one target identity.  Both poset trackers do better than the Kalman tracker.  Over the rest of the plot, the simple weighting appears to be roughly similar in performance (possibly slightly worse) than the Kalman filter.

\subsection{Indoor sonar}

The sonar data contains no true identity information, and has relatively high clutter.  The sonar transmitter consisted of a pair of inexpensive 10 W speakers connected to a laptop computer, while the same laptop recorded echos from a wired portable microphone such as those used in lecture halls.  The recording sample rate was 44.1 kHz and the pulse repetition rate was 29.4 Hz (1500 samples per pulse).  Each individual pulse consisted of an impulse whose transmitted bandwidth was around 8--10 kHz.  The collection captured around 200 pulses.

We collected our data in the Jacobs Fitness Center on the American University campus.  The general setup is shown in Figure \ref{fig:sonarsetup}.  All scatterers in the room were held fixed, while the microphone was slowly dragged along the floor toward the laptop by pulling the wire connecting the microphone to the laptop.  In this way, the range of prominent scatterers in the the room drifted, giving the appearance of many moving targets.

\begin{figure}
  \begin{center}
    \includegraphics[width=3.5in]{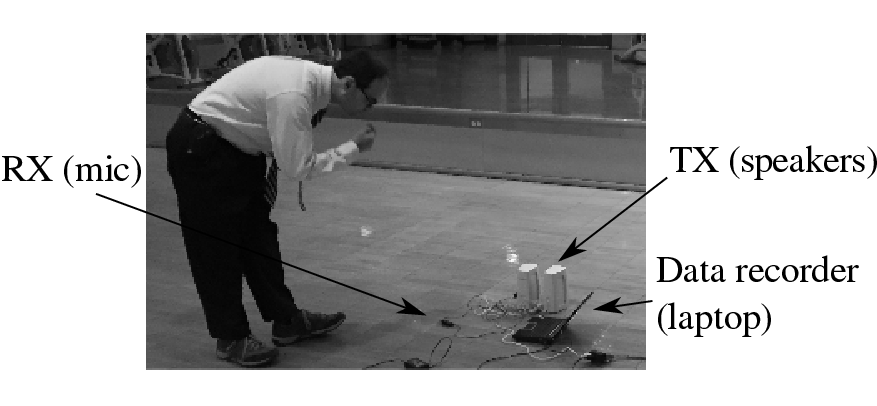}
    \caption{Experimental setup for the sonar data collection}
    \label{fig:sonarsetup}
  \end{center}
\end{figure}

After the data were collected, the pulses were aligned so that the strongest echo in each pulse interval occurred at the first sample.  Echos were identified using a standard constant false alarm rate detector with a window of 100 samples with a 10 sample guard.  When a connected set of observations occurred, only the centroid of this connected set was recorded as a detection.  The centroid was computed using the received signal strength as the density.  Large connected sets of observations occurred frequently, but were especially prominent for strong echos.  The resulting set of sonar detections and tracks is shown in Figure \ref{fig:sonar_tracks_07b}.

Since the sonar echos after preprocessing are indistinguishable, we cannot verify that sufficient sample rates (according to Propositions \ref{prop:timeline_separation} or \ref{prop:search_region_size}) were obtained.  We also cannot obtain a model like Figure \ref{fig:airdata_neighborhood}, so the maximum speed was estimated by hand from the apparent slope of the chain of observations in the middle of the collection.  This was found to be about 4 range cells per pulse.  The apparent timelines of various ``targets'' cross both in the lower left portion of Figure \ref{fig:sonar_tracks_07b} and around pulse 75 at the top of the Figure.  This means that the sample rate bound in Proposition \ref{prop:timeline_separation} is certainly not satisfied, as the appropriate sampling interval would be zero!

The set of detections was then fed into the same multi-target Kalman tracker algorithm as in Section \ref{sec:airtraffic} for comparison with Algorithm \ref{alg:posettrack}, though with different parameters.  Specifically, covariances were tuned for optimal track quality and stale tracks were dropped after 50 pulses.  The output of this algorithm appears in the top frame of Figure \ref{fig:sonar_tracks_07b}.  It is immediately clear that considerable ambiguity exists, especially in close range (near the origin).

\begin{figure}
  \begin{center}
    \includegraphics[width=4in]{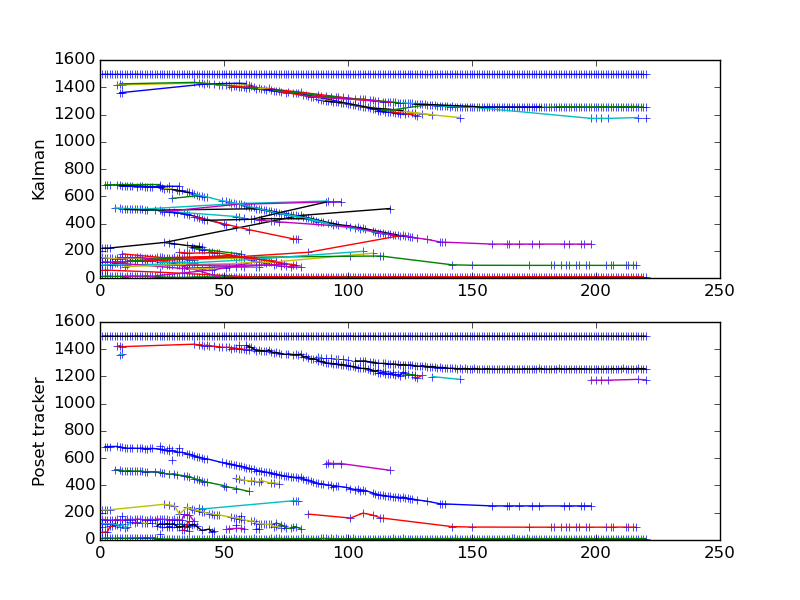}
    \caption{Tracks produced by a standard multi-target Kalman filter (top) and Algorithm \ref{alg:posettrack} (bottom).  The horizontal axis represents pulse number (slow time) while the vertical axis represents sample number (bistatic range)}
    \label{fig:sonar_tracks_07b}
  \end{center}
\end{figure}

To apply Algorithm \ref{alg:posettrack}, we require the design of a constraint function.
The constraint function employed a cutoff time cutoff after 50 pulses, a displacement cutoff of 75 range cells, and a maximum speed of 50 range cells per pulse (which was well above the maximum observed: 4 range cells per pulse).  The weights were generated solely from the spatial distance (in range) between observations.
The output of Algorithm \ref{alg:posettrack} is shown in the bottom frame of Figure \ref{fig:sonar_tracks_07b}.  

Using Algorithm \ref{alg:posettrack}, there is considerably less confusion than with the Kalman filter about which targets go with which tracks.  Not all ambiguity has been removed, though, most notably the cyan track.

\section{Conclusions and Future research}
\label{sec:conclusions}

This article presents a practical, first-principles approach to multi-object tracking that supports feature-poor targets obeying set-valued dynamical laws.  We proved performance guarantees based on the sampling rate of the tracker, and demonstrated the feasibility of an algorithm based on our theory on laboratory and fielded datasets.

Both of the sample interval bounds---for search region size and loss of custody---are upper bounds, which future work could attempt to improve.  It is likely that constraint functions can be designed from the transition probability density function and could tighten these bounds, though this may be computationally quite challenging.  Additional improvements can also be made to the weighting, especially given that we found that the tracker performance depends strongly on the weighting.

\section*{Acknowledgments}
The authors would like to thank Prof. Robert Ghrist for his guidance throughout the project, and especially for his feedback during the preparation of this article.  We also thank David Lipsky for algorithmic discussions and assistance with interpreting the results.

\bibliographystyle{plain}
\bibliography{posettrack_bib}

\begin{IEEEbiography}
[{\includegraphics[width=1in]{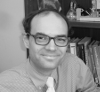}}]
{Michael Robinson} is a professor of mathematics and statistics at American University. He is interested in signal processing, dynamics, and applications of topology.  His current projects involve topological approaches to tracking, communication network analysis, sonar target recognition, and data fusion. 
\end{IEEEbiography}

\begin{IEEEbiography}
[{\includegraphics[width=1in]{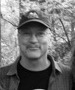}}]
{Michael Stein} is the co-founder of Tesseract Analytics, LLC, a small company that develops advanced analytics for government and commercial clients.  His interests include developing algorithms that apply topological models for solving hard problems in tracking, social network analysis, and text understanding.  Recent projects have included topological approaches to tracking, extraction of concepts and trends from news stories and analyst reports, and automated support for sharing and visualization of political, social, and economic data about countries.
\end{IEEEbiography}

\begin{IEEEbiography}
[{\includegraphics[width=1in]{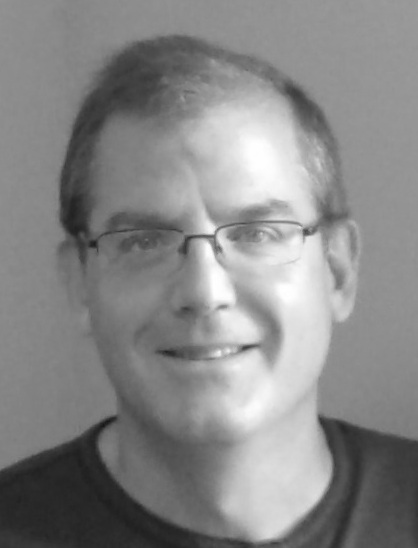}}]
{Henry Owen} is the owner of two technology companies, HS Owen LLC and Slingshot Analytics, Inc.  His companies  focus on various aspects of radio frequency (RF) propagation, sensing, communication, and navigation.  His current projects include RF propagation inversion and tracking and geolocation problems using novel data sources.
\end{IEEEbiography}

\end{document}